\documentclass[3p,preprint,11pt]{elsarticle}

\makeatletter
\def\ps@pprintTitle{%
 \let\@oddhead\@empty
 \let\@evenhead\@empty
 \def\@oddfoot{\centerline{\thepage}}%
 \let\@evenfoot\@oddfoot}
\makeatother

\usepackage{graphicx}
\usepackage{amsmath}
\usepackage{amsmath}
\usepackage{amssymb}
\usepackage{amsthm}
\usepackage{color}
\usepackage{subfigure}
\usepackage{url}
\usepackage{algorithm}
\usepackage{algorithmic}

\newcommand{\seclabel}[1]{\label{sec:#1}}
\newcommand{\secref}[1]{Section~\ref{sec:#1}}

\newcommand{\eqlabel}[1]{\label{eq:#1}}
\newcommand{\Eqref}[1]{Equation~\eqref{eq:#1}}
\newcommand{\figlabel}[1]{\label{fig:#1}}
\newcommand{\figref}[1]{Figure~\ref{fig:#1}}
\newcommand{\rref}[1]{Reference~\cite{#1}}

\newcommand{\thmlabel}[1]{\label{thm:#1}}
\newcommand{\thmref}[1]{Theorem~\ref{thm:#1}}

\newcommand{\lemlabel}[1]{\label{lem:#1}}
\newcommand{\lemref}[1]{Lemma~\ref{lem:#1}}
\newcommand{\remlabel}[1]{\label{rem:#1}}
\newcommand{\remref}[1]{Remark~\ref{rem:#1}}

\newtheorem{theorem}{Theorem}[section]
\newtheorem{lemma}[theorem]{Lemma}

\theoremstyle{remark}
\newtheorem{remark}{Remark}[section]

\theoremstyle{definition}

\newcommand{\E}{\mathcal{E}}
\newcommand{\F}{\mathcal{F}}
\DeclareMathOperator{\Mod}{Mod}

\newcommand{\N}{\mathcal{N}}
\newcommand{\eff}{\mathrm{eff}}

\begin{document}

\begin{frontmatter}

  \title{Modulus on graphs as a generalization of standard graph
    theoretic quantities}

\author[ksu]{Nathan Albin}
\ead{albin@math.ksu.edu}

\author[reu]{Megan Brunner}

\author[reu]{Roberto Perez}

\author[ksu]{Pietro Poggi-Corradini}
\ead{pietro@math.ksu.edu}

\author[reu]{Natalie Wiens}

\address[ksu]{Department of Mathematics, Kansas State University, 138
  Cardwell Hall, Manhattan, KS 66506}

\address[reu]{Research supported at Kansas State University by NSF REU
  Grant No.~126287~}

\begin{abstract}
  This paper presents new results for the modulus of families of walks
  on a graph---a discrete analog of the modulus of curve families due
  to Beurling and Ahlfors.  Particular attention is paid to the
  dependence of the modulus on its parameters.  Modulus is shown to
  generalize (and interpolate among) three important quantities in
  graph theory: shortest path, effective resistance, and max-flow or
  min-cut.
\end{abstract}

  
\end{frontmatter}

\section{Introduction}\seclabel{intro}

Beginning with an example, consider the simple, undirected graph shown
in \figref{example}.  Such a graph may be used to model, for example,
a contact network for use in predicting the spread of disease, where
each vertex indicates an individual in a population and each edge
indicates that the two corresponding individuals have direct contact
with each other.  The disease is then assumed to spread throughout the
network according to some stochastic process.  An important question
in such a model is the following.  How easily can individual $s$
transmit the disease to individual $t$?  In a sense, we are asking for
a distance-like measure between $s$ and $t$.

One obvious approach would be to define distance in the classical
\emph{graph distance} sense: the distance between $s$ and $t$ is the
length of the shortest path connecting them.  In this case, $s$ and
$t$ would be at distance $1$ from each other.  For this particular
application, however, this definition seems inappropriate.  In graph
distance, $s$ is equidistant from $u$ and $t$.  However, any
reasonable disease model would indicate that a disease originating in
$s$ is more likely to spread to $t$ than to $u$; there are more
avenues for transmission from $s$ to $t$.

Taking this shortcoming into account, we might try to define proximity
through the number of distinct transmission pathways.  That is, we
could count the maximum number of pairwise edge-disjoint paths from
$s$ to $t$: 3 in this case.  By Menger's theorem, this is equivalent
to the minimum number of edges that can be removed from the graph in
order to separate $s$ from $t$.  This \emph{min-cut} sense of
proximity addresses the earlier concern with graph distance, since $s$
is now three times closer to $t$ than to $u$.  However, this measure
is also lacking; $t$ is now equally close to both $s$ and $v$, which
again seems counterintuitive for the application.

A more recent description of distance that bridges the gap between
these two extremes is the \emph{resistance
  distance}, see for instance \cite{klein_resistance_1993}.  The resistance distance
between $s$ and $t$ is defined as the effective resistance between the
corresponding nodes in an electrical resistor network wherein the
vertices of the graph correspond to junctions and the edges to
resistors of unit resistance.  In resistance distance, $s$ and $t$ are
$0.454$ units apart, $s$ and $u$ are $1$ unit apart, and $t$ and $v$
are $1.918$ units apart.

\begin{figure}
  \centering
  \includegraphics{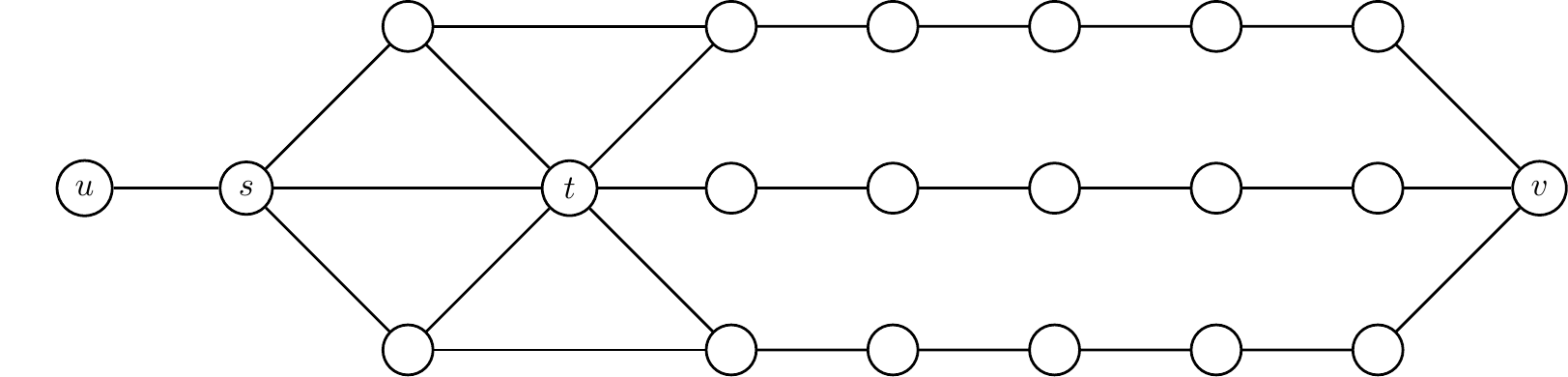}
  \caption{Example graph for \secref{intro}.}
  \figlabel{example}
\end{figure}

Although resistance distance does give a more satisfying answer for
the contact network application, a number of interesting questions
remain.  For example, although the shortest path and min-cut distances
have straightforward extensions to directed graphs, effective
resistance does not. So, what is the natural analog of resistance
distance for directed graphs?  Or consider the following related type
of question for the graph in \figref{example}.  It is clear that $s$
plays a crucial role in the transmission of disease between $u$ and
$t$.  But, how might we quantify the role that $t$ plays in
transmission between $s$ and $v$?  In other words, how rich is the set
of pathways connecting $s$ to $v$ through $t$?  And how does this
compare to the set of all pathways connecting $s$ to $v$?  Or for that
matter, given any interesting collection of walks on the graph, how
does one quantify the richness of the collection?

A potential answer to these questions comes in the form of the
modulus---a discrete analog of the modulus of curve families dating at
least as far back as Beurling \cite{beurling1989} and Ahlfors
\cite{ahlfors1982}.  Although discrete versions of the modulus have
been introduced before, the goal of the present work is to make
explicit the connection to more traditional graph theoretic concepts:

\begin{itemize}
\item We treat the modulus in a general framework, accommodating
  directed or undirected graphs, with or without edge weights
  (\secref{modulus}).  Although this is not entirely new (see
  \secref{related} for a description of related work), we do adopt a
  particular viewpoint that lends itself well to graph theoretic
  interpretations of the modulus.
\item After framing the modulus as a convex optimization problem, we
  derive and interpret a dual formulation of the modulus
  (\secref{lagrange}).  This is similar to the duality of the max-flow
  and min-cut problems, but much more general than has been done before.
\item Using the primal-dual formulation, we establish three
  connections to graph theory, namely that computing the modulus in
  certain circumstances is equivalent to computing either the graph
  distance, min-cut distance, or resistance distance
  (\secref{graph-theory}).  Moreover, we prove that the modulus (which
  depends on a continuous parameter $p$) effectively interpolates
  among these three concepts (\secref{p-dependence}).
\item In \secref{modulus}, the modulus is defined through real-valued
  densities on the edge set of the graph.  We show that these
  densities generalize the concept of potential drop across resistors
  in the computation of resistance distance (see
  \thmref{density-potential}) and analyze the behavior of these
  densities (\secref{rho-properties}).
\end{itemize}

\section{The modulus}\seclabel{modulus}

\subsection{Graphs, walks, and families of walks}\seclabel{walks}

In what follows, $G$ is a graph with vertex set $V$ and edge set $E$.
The sets of vertices and edges are assumed finite, with $n=|V|$
vertices and $m=|E|$ edges.  The graph may be directed or undirected.
In the former case, $E$ is represented as a subset of ordered pairs of
vertices, in the latter as a subset of unordered pairs.  We adopt the
notation $e=(x,y)$ to represent an edge in either case, with the
understanding that $(x,y)=(y,x)$ in the undirected case.  For any
function $f:E\to\mathbb{R}$ and any edge $e=(x,y)$ we will use the
notation $f(e)$ and $f(x,y)$ interchangeably.

The discussion in this paper is restricted to simple graphs, though
there are natural extensions to multigraphs and to other more general
objects. In general, $G$ is assumed to be weighted, with weight
function $\sigma:E\to(0,\infty)$.  A graph is called \emph{unweighted}
if $\sigma\equiv 1$.  For the purposes of this paper, then, a graph is
defined through the triple $G=(V,E,\sigma)$.

Using the standard terminology, a \emph{walk} on $G$ is represented as
\begin{itemize}
\item a finite string of vertices $v_1v_2v_3 \ldots v_{r+1} $ in
  $V$, with the property that $(v_i,v_{i+1})\in E$ for
  $i=1,2,\ldots,r$ or, equivalently,
\item the corresponding string of edges $e_1e_2\ldots e_r$, where $e_i
  =(v_i,v_{i+1})$ for $i=1,2,\ldots,r$.
\end{itemize}
In either case, we require that $r\ge 1$, so that a walk is considered
to traverse at least one edge in the graph.

To each walk $\gamma=e_1e_2\ldots e_r$ is associated the \emph{graph
  length}, $\ell(\gamma) := r$.  That is, $\ell(\gamma)$ counts the
number of ``hops'' made in traversing $\gamma$.  More generally, given
any edge density $\rho:E\to\mathbb{R}$, the \emph{$\rho$-length} of
$\gamma$, $\ell_\rho(\gamma)$, is defined as
\begin{equation*}
  \ell_\rho(\gamma) := \sum_{i=1}^r\rho(e_i).
\end{equation*}
The graph length of $\gamma$ coincides with the $\rho$-length for the
particular choice $\rho\equiv 1$.

Throughout this paper, $\Gamma$ represents an arbitrary nonempty
family of walks on $G$.  A particularly useful class of walk family is
the \emph{connecting family}, denoted $\Gamma(s,t)$, of walks
originating at a vertex $s\in V$ and terminating at a distinct vertex
$t\in V\setminus\{s\}$.

\subsection{Densities and admissibility}

Given a walk family $\Gamma$ and an edge density
$\rho:E\to\mathbb{R}$, the graph length and $\rho$-length of $\Gamma$
are defined as
\begin{equation}\eqlabel{ell-Gamma}
  \ell(\Gamma) := \inf_{\gamma\in\Gamma}\ell(\gamma)\qquad\text{and}\qquad
  \ell_\rho(\Gamma) := \inf_{\gamma\in\Gamma}\ell_\rho(\gamma)
\end{equation}
respectively.  

The family $\Gamma$ determines three sets of densities called the
\emph{admissible set} $A(\Gamma)$, the \emph{relaxed admissible set}
$A'(\Gamma)$, and the \emph{restricted admissible set} $A^*(\Gamma)$
defined as follows.
\begin{equation*}
  \begin{split}
    A'(\Gamma) &:= \{\rho:E\to\mathbb{R} : \ell_\rho(\Gamma)\ge 1 \},\\
    A(\Gamma) &:= \{\rho:E\to\mathbb{R} : \ell_\rho(\Gamma)\ge 1,\quad 0\le\rho \},\\
    A^*(\Gamma) &:= \{\rho:E\to\mathbb{R} : \ell_\rho(\Gamma)\ge 1,\quad 0\le\rho\le 1 \}.\\
  \end{split}
\end{equation*}
Note that these definitions are made in the order of decreasing size:
$A^*(\Gamma)\subset A(\Gamma)\subset A'(\Gamma)$.  The reason for
defining three different admissibility sets will be apparent by
\remref{A-sets}.

\subsection{Energy and modulus}
Given a real parameter $p\ge 1$ or $p=\infty$, the \emph{$p$-energy}
of a density $\rho$ is
\begin{equation*}
  \E_p(\rho) :=
  \begin{cases}
    \sum\limits_{e\in E}\sigma(e)|\rho(e)|^p & \text{if }1\le p < \infty\\
    \max\limits_{e\in E}|\rho(e)| & \text{if } p = \infty
  \end{cases}
\end{equation*}
The definition for $\E_\infty$ is consistent in the sense that
\begin{equation*}
  \forall\rho:E\to\mathbb{R}\qquad \lim_{p\to\infty}\E_p(\rho)^{1/p} = \E_\infty(\rho).
\end{equation*}
For $1\le p\le\infty$, the \emph{$p$-modulus} of $\Gamma$ is defined
as
\begin{equation*}
  \Mod_p(\Gamma) := \inf_{\rho\in A(\Gamma)} \E_p(\rho).
\end{equation*}
(The modulus of an empty family $\Mod_p(\varnothing)$ is defined to be
zero, since the choice $\rho\equiv 0$ is trivially admissible.)

\subsection{Extremal densities}

An edge density $\rho\in A(\Gamma)$ is called \emph{extremal} for a
given family $\Gamma$ and a given $p$ if
\begin{equation*}
  \Mod_p(\Gamma) = \E_p(\rho).
\end{equation*}
The notation $\rho^*$ or $\rho_p$ is frequently used to denote an
extremal density, with the latter used in order to make the dependence
on $p$ explicit.

\begin{lemma}\lemlabel{rho-star}
  Given $\Gamma$ and $1\le p\le\infty$, there exists an extremal
  density $\rho^*$.  If $1<p<\infty$, then $\rho^*$ is unique.
\end{lemma}
\begin{proof}
  For any $1\le p<\infty$, the function $\rho\mapsto\E_p(\rho)^{1/p}$
  is a norm on the $m$-dimensional space of edge densities, as is the
  function $\rho\mapsto\E_\infty(\rho)$.  Moreover, the set
  $A(\Gamma)$ is closed and convex.  Thus, the modulus problem can be
  restated as the problem of finding a point in a closed convex subset
  of $\mathbb{R}^m$ that is closest to the origin in a particular
  norm.  Such a point always exists, and is unique provided the norm
  is strictly convex, as it is in the cases $1<p<\infty$.
\end{proof}

The following observation, which shows that the nonnegativity
constraint on $\rho$ need not be explicitly enforced, will prove
useful in \secref{lagrange}.
\begin{lemma}\lemlabel{rho-bounds}
  Let $\Gamma$ and $1\le p\le\infty$ be given.  Then
  \begin{equation}\eqlabel{mod-A-prime}
    \Mod_p(\Gamma) = \inf_{\rho\in A'(\Gamma)}\E_p(\rho).
  \end{equation}
  An extremal $\rho^*\in A'(\Gamma)$ always exists for the relaxed
  problem, and is unique if $1<p<\infty$.  Moreover, for $1\le
  p<\infty$ any such extremal $\rho^*$ satisfies the bounds $0\le
  \rho^*\le 1$, and thus $\rho^*\in A^*(\Gamma)$. When $p=\infty$
  there exists at least one extremal $\rho^*\in A^*(\Gamma)$.
\end{lemma}
\begin{proof}
  Existence and uniqueness (for $1<p<\infty$) of a minimizer for the
  infimum in \Eqref{mod-A-prime} are proved as in \lemref{rho-star}.

  Now, suppose $\rho\in A'(\Gamma)$ takes a negative value on some
  edge $e'\in E$, and define $\rho^+$ as
  $\rho^+(e)=\max\{\rho(e),0\}$.  Since $\ell_{\rho^+}(\gamma)\ge
  \ell_{\rho}(\gamma)\ge 1$ for any $\gamma\in\Gamma$, $\rho^+$ is
  also admissible.  Moreover, on every edge $e\in E$, $|\rho^+(e)|\le
  |\rho(e)|$ and $|\rho^+(e')|=0<|\rho(e')|$.  For $1\le p<\infty$,
  this implies $\E_p(\rho^+)<\E_p(\rho)$, showing that any minimizing
  $\rho^*$ is in $A(\Gamma)$.  For the $p=\infty$ case, the inequality
  $\E_\infty(\rho^+)\le\E_\infty(\rho)$ holds, but need not be strict.
  However, a minimizing $\rho^*$ can always be found in $A(\Gamma)$.

  Next, suppose $\rho\in A(\Gamma)$ and that $\rho(e')>1$ on some edge
  $e'\in E$.  Define $\rho'(e) = \min\{\rho(e),1\}$.  For any
  $\gamma\in\Gamma$, either $\ell_{\rho'}(\gamma) =
  \ell_\rho(\gamma)\ge 1$, or $\gamma$ traverses an edge where $\rho'$
  takes the value $1$.  In the latter case, the inequality
  $\ell_{\rho'}(\gamma)\ge 1$ is trivial, so $\rho'\in A(\Gamma)$.
  Since $\rho'(e)\le\rho(e)$ on all edges, and since the inequality is
  strict on at least one edge, we have $\E_p(\rho')<\E_p(\rho)$ (for
  $1\le p\le\infty$).  This proves that $\rho^*\le 1$.
\end{proof}

\begin{remark}\remlabel{A-sets}
  A consequence of \lemref{rho-bounds} is that the admissibility set
  $A(\Gamma)$ in the definition of modulus can be replaced with either
  a more relaxed set, $A'(\Gamma)$, or a more restricted set,
  $A^*(\Gamma)$, without changing the outcome.  That is,
  \begin{equation*}
    \Mod_p(\Gamma) := \inf_{\rho\in A(\Gamma)}\E_p(\rho) = 
    \inf_{\rho\in A'(\Gamma)}\E_p(\rho) = 
    \inf_{\rho\in A^*(\Gamma)}\E_p(\rho)\quad
    \text{for all }1\le p\le\infty.
  \end{equation*}
  A convenient choice of
  admissibility set can simplify proofs in certain circumstances.  For
  example, when deriving the Lagrange dual problem in the case $1<p<\infty$
  (\secref{lagrange}) it is better to use the most relaxed set
  $A'(\Gamma)$.  The proof of \thmref{Gamma-star} utilizes the lower
  bound in $A(\Gamma)$, while the bounds of \secref{p-dependence} make
  use of both upper and lower bounds in $A^*(\Gamma)$.
\end{remark}

\subsection{An instructive example}\seclabel{example}

As described in the introduction, the $p$-modulus is a generalization
of three fundamental quantities in graph theory: shortest path length,
effective conductance, and min-cut.  Before proving these facts, it is
useful to explore an example.  Consider the (unweighted and
undirected) graph shown in \figref{paths}---comprising $k$ parallel
simple paths of $\ell$ hops connecting node $s$ to node $t$---with
$\Gamma=\Gamma(s,t)$, the connecting family of walks defined at the
end of \secref{walks}.

\begin{figure}
  \centering
  \includegraphics{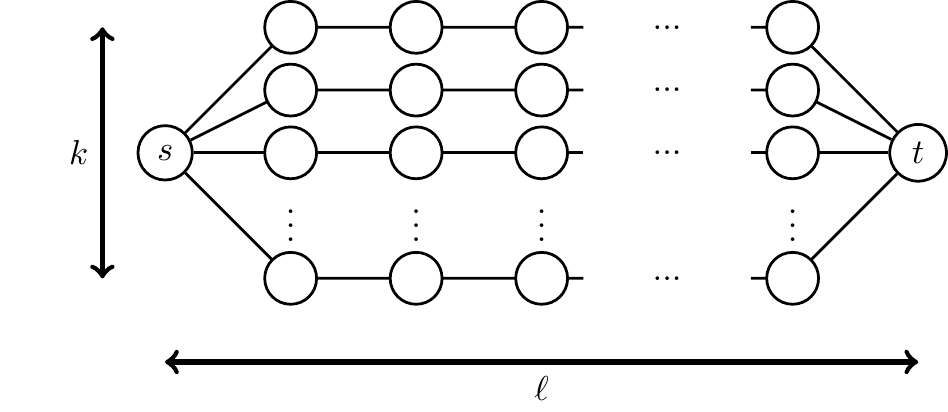}
  \caption{A graph consisting of $k$ simple paths with $\ell$ hops
    each, connecting $s$ to $t$.}
  \figlabel{paths}
\end{figure}

It is a straightforward exercise to show that the choice
$\rho^*\equiv\frac{1}{\ell}$ is extremal for all $1\le p \le \infty$.
Thus, for $1\le p <\infty$, the modulus of $\Gamma$ is
\begin{equation*}
  \Mod_p(\Gamma) = \E_p(\rho^*) = k\ell\left(\frac{1}{\ell}\right)^p
  = \frac{k}{\ell^{p-1}},
\end{equation*}
and for $p=\infty$,
\begin{equation*}
  \Mod_\infty(\Gamma) = \E_\infty(\rho^*) = \frac{1}{\ell}.
\end{equation*}

There are several interesting observations to be made.  First, as a
function of $p$, the modulus is continuous, monotone decreasing, and
decays to zero like $\ell^{-p}$ as $p\to\infty$ if $\ell>1$.  All
three of these properties are generic, as shown in
\secref{p-dependence}.

Moreover, for $1<p<\infty$, the modulus depends on both the number of
distinct walks, $k$, and on the length of these walks, $\ell$.  The
modulus is large for $\Gamma$ containing many, short walks, and small
for $\Gamma$ containing few, long walks.  For $p$ near $1$, the
modulus is much more sensitive to changes in $k$ than to changes in
$\ell$, while, for $p$ very large, the modulus is much more sensitive
to changes in $\ell$.  In the extreme case $p=1$, the modulus loses
all dependence on $\ell$: a large family of walks has large modulus
regardless of their lengths.  On the other hand, when $p=\infty$ the
modulus loses all dependence on $k$: a family's modulus depends only
on the length of the shortest walk and not at all on the number of
walks.  The case $p=2$ is also interesting.  In this case, the modulus
is $k/\ell$, which is the effective conductance of the graph viewed as
a resistor network between $s$ and $t$ with unit resistors placed on
each edge.

Again, these are general properties.  \secref{graph-theory} proves
that, in a specific sense, the $1$-modulus is a generalization of the
min-cut problem, that the $\infty$-modulus is always equal to the
reciprocal of the length of the shortest walk, and that the
$2$-modulus is a generalization of effective conductance.  The results
of \secref{p-dependence}, then, show that $p$ can be thought of as a
tuning parameter, which adjusts the balance of importance of ``many
walks'' with ``short walks'' in the modulus.

\section{Lagrange dual formulation}\seclabel{lagrange}

The $p$-modulus problem is a convex optimization problem: it involves
minimizing the convex function $\E_p$ over the convex set $A(\Gamma)$
defined by a set of linear inequalities.  Furthermore, it was shown
in~\cite{Albin} that this set of inequalities can be assumed finite;
more precisely, the following theorem is true.

\begin{theorem}[\cite{Albin}]\thmlabel{Gamma-star}
  Let $\Gamma$ be a given family of walks on a graph.  There exists a
  finite subfamily $\Gamma^*\subseteq\Gamma$ such that
  $A(\Gamma^*)=A(\Gamma)$.
\end{theorem}

Such a finite subfamily is called an \emph{essential subfamily} for
the modulus problem.  Ensuring that $\ell_\rho(\gamma)\ge 1$ is true
for all walks in the essential subfamily is equivalent to ensuring the
inequality for the entire family.  Thus, even when a described set of
walks (e.g., the connecting family $\Gamma(s,t)$) contains infinitely
many members, it can always be assumed that $\Gamma$ has been replaced
by a finite essential subfamily (e.g., the simple paths from $s$ to
$t$ when $\Gamma=\Gamma(s,t)$).  With this assumption, the modulus
problem becomes an ordinary convex program, in the sense
of~\cite[Sec.~28]{Rockafellar1970}, taking the form
\begin{equation}\eqlabel{cvx-general}
  \begin{split}
    \text{minimize}\quad&\E_p(\rho)\\
    \text{subject to}\quad& \sum_{e\in E}\N(\gamma,e)\rho(e)\ge
    1\quad\forall\gamma\in\Gamma,\\
    & \rho(e)\ge 0\quad\forall e\in E,
  \end{split}
\end{equation}
where $\N(\gamma,e)$ is the $|\Gamma|\times|E|$ matrix whose
$(\gamma,e)$ entry contains the number of times walk $\gamma$
traverses edge $e$.

Using standard techniques from convex optimization, we define the
Lagrangian for the problem as
\begin{equation}\eqlabel{general-lagrangian}
  L_p(\rho,\lambda,\mu) := \E_p(\rho) + 
  \sum_{\gamma\in\Gamma}\lambda(\gamma)\left(1-\sum_{e\in E}\N(\gamma,e)\rho(e)\right)
  -\sum_{e\in E}\mu(e)\rho(e),
\end{equation}
where $\lambda:\Gamma\to[0,\infty)$ and $\mu:E\to[0,\infty)$ are the
Lagrange dual variables associated with the admissibility constraints.
The usual \emph{weak duality} statement takes the following form.
\begin{lemma}\lemlabel{weak-duality}
  Let $\Gamma$ and $p$ be given.  Then
  \begin{equation*}
    \Mod_p(\Gamma) = \inf_{\rho:E\to\mathbb{R}}\sup_{\substack{\lambda:\Gamma\to[0,\infty)\\\mu: E\to[0,\infty)}}L_p(\rho,\lambda)
    \ge \sup_{\substack{\lambda:\Gamma\to[0,\infty)\\\mu: E\to[0,\infty)}}\inf_{\rho:E\to\mathbb{R}}L_p(\rho,\lambda)
    = \sup_{\substack{\lambda:\Gamma\to[0,\infty)\\\mu: E\to[0,\infty)}}\F_p(\lambda,\mu)
  \end{equation*}
  where
  \begin{equation*}
    \F_p(\lambda,\mu) := \inf_{\rho:E\to\mathbb{R}}L_p(\rho,\lambda,\mu).
  \end{equation*}
\end{lemma}
Given any $\lambda:\Gamma\to[0,\infty)$ and any $\mu:E\to[0,\infty)$,
weak duality implies that $\Mod_p(\Gamma)\ge\F_p(\lambda,\mu)$.  The
problem of maximizing this lower bound over all nonnegative $\lambda$
and $\mu$ is the \emph{Lagrange dual problem} to the modulus:
\begin{equation*}
  \begin{split}
    \text{maximize}\quad&\F_p(\lambda,\mu)\\
    \text{subject to}\quad& \lambda(\gamma)\ge
    0\quad\forall\gamma\in\Gamma,\\
    &\mu(e)\ge 0 \quad\forall e\in E.
  \end{split}
\end{equation*}
\begin{remark}\remlabel{strong-duality}
  Because the inequality constraints in~\eqref{eq:cvx-general} are
  affine, it follows from Theorem~28.2 of~\cite{Rockafellar1970} that
  \emph{strong duality} holds.  That is, there exists a saddle point
  $(\rho^*,\lambda^*,\mu^*)$ of the Lagrangian.  This is an important
  property of the modulus problem, especially in the context of
  numerical approximation, and was explored in a specialized setting
  in~\cite{Albin}.  For the present paper, the weak duality result is
  usually sufficient.
\end{remark}

When working with the dual problem, it is convenient to have a more
explicit formulation.  We derive such formulations in three different
cases: $1<p<\infty$, $p=1$ and $p=\infty$.

\subsection{The case $1<p<\infty$}

When $1<p<\infty$, \lemref{rho-bounds} implies that the $\rho\ge 0$
constraints need not be explicitly enforced in~\eqref{eq:cvx-general},
and can thus be removed.  The simplified Lagrangian (omitting the dual
variables $\mu$ as they are no longer necessary) takes the form
\begin{equation*}
  L_p(\rho,\lambda) := \sum_{e\in E}\sigma(e)|\rho(e)|^p 
  + \sum_{\gamma\in\Gamma}\lambda(\gamma)\left(1-\sum_{e\in E}\N(\gamma,e)\rho(e)\right),
\end{equation*}
which is differentiable in each $\rho(e)$, since $p>1$.  Thus, the value of the
dual energy
\begin{equation*}
  \F_p(\lambda) = \inf_{\rho:E\to\mathbb{R}}\left\{
    \sum_{e\in E}\sigma(e)|\rho(e)|^p 
    + \sum_{\gamma\in\Gamma}\lambda(\gamma)\left(1-\sum_{e\in E}\N(\gamma,e)\rho(e)\right)    
  \right\}
\end{equation*}
can be found by solving $\nabla_\rho L_p=0$.  The optimal density,
$\rho_\lambda$, for a given $\lambda$ must satisfy the stationarity
condition
\begin{equation*}
  p\sigma(e)\rho_\lambda(e)|\rho_\lambda(e)|^{p-2} - \sum_{\gamma\in\Gamma}\N(\gamma,e)\lambda(\gamma) = 0\quad\forall e\in E.
\end{equation*}
Since $\lambda\ge 0$ and $\N$ has non-negative entries, the above
equation implies that $\rho_\lambda\ge 0$ and so
\begin{equation}\eqlabel{rho-lambda}
  \rho_\lambda(e) = \left(\frac{1}{p\sigma(e)}\sum_{\gamma\in\Gamma}\N(\gamma,e)\lambda(\gamma)\right)^{\frac{1}{p-1}}.
\end{equation}
After some simplification, this produces the following formula for
$\F_p(\lambda)$.
\begin{equation*}
  \F_p(\lambda) = L(\rho_\lambda,\lambda) = 
  \sum_{\gamma\in\Gamma}\lambda(\gamma) - (p-1)\sum_{e\in E}\sigma(e)\left(
    \frac{1}{p\sigma(e)}\sum_{\gamma\in\Gamma}\N(\gamma,e)\lambda(\gamma)
  \right)^{\frac{p}{p-1}}.
\end{equation*}
We conclude that the Lagrange dual optimization problem can be written
as
\begin{equation}\eqlabel{dual-p}
  \begin{split}
    \text{maximize}\quad&
    \sum_{\gamma\in\Gamma}\lambda(\gamma) - (p-1)\sum_{e\in E}\sigma(e)\left(
      \frac{1}{p\sigma(e)}\sum_{\gamma\in\Gamma}\N(\gamma,e)\lambda(\gamma)
    \right)^{\frac{p}{p-1}}\\
    \text{subject to}\quad& \lambda(\gamma)\ge 0\quad\forall\gamma\in\Gamma.
  \end{split}
\end{equation}

\begin{remark}
  Because strong duality holds (see \remref{strong-duality}),
  solving~\eqref{eq:dual-p} is equivalent to computing the modulus.
  This provides an interesting interpretation of the modulus: rather
  than assigning values to edges (through $\rho$), we instead assign
  values to walks in $\Gamma$ (through $\lambda$) in such a way as to
  maximize the dual energy.  \Eqref{rho-lambda} provides the formula
  for determining the extremal density $\rho^*$ from an extremal
  $\lambda^*$. On the other hand, the extremal $\lambda^*$ are
  typically not unique.
\end{remark}

\subsection{The case $p=1$}
In the case $p=1$, the general Lagrangian in
\Eqref{general-lagrangian} is not differentiable, so the preceding
argument does not apply.  This issue can be circumvented by replacing
the $1$-modulus optimization problem with an equivalent problem:
\begin{equation*}
  \begin{split}
    \text{minimize}\quad&\sum_{e\in E}\sigma(e)\rho(e)\\
    \text{subject to}\quad& \sum_{e\in E}\N(\gamma,e)\rho(e)\ge
    1\quad\forall\gamma\in\Gamma\\
    & \rho(e)\ge 0\quad\forall e\in E.
  \end{split}  
\end{equation*}
Dropping the absolute values from the energy has no effect on the
outcome, since the values of $\rho$ are required to be nonnegative.
However, without the absolute values, the energy is now differentiable
for general densities $\rho:E\to\mathbb{R}$.

Introducing the dual variables $\lambda:\Gamma\to[0,\infty)$ and
$\mu:E\to[0,\infty)$, the Lagrangian for this problem is
\begin{equation*}
  \begin{split}
    L_1(\rho,\lambda,\mu) &:= \sum_{e\in E}\sigma(e)\rho(e) +
    \sum_{\gamma\in\Gamma}\lambda(\gamma)\left(1-\sum_{e\in
        E}\N(\gamma,e)\rho(e)\right) - \sum_{e\in E}\mu(e)\rho(e)\\
    &= \sum_{\gamma\in\Gamma}\lambda(\gamma) + \sum_{e\in
      E}\rho(e)\left( \sigma(e) -
      \sum_{\gamma\in\Gamma}\N(\gamma,e)\lambda(\gamma) - \mu(e)
    \right).
  \end{split}
\end{equation*}
Since the Lagrangian is affine in $\rho(e)$, it is clear that
\begin{equation*}
  \F_1(\lambda,\mu) := \inf_{\rho:E\to\mathbb{R}} L_1(\rho,\lambda,\mu)
  =
  \begin{cases}
    \sum\limits_{\gamma\in\Gamma}\lambda(\gamma) &\text{if}\quad \sigma(e) -
    \sum\limits_{\gamma\in\Gamma}\N(\gamma,e)\lambda(\gamma) - \mu(e) = 0\quad\forall e\in E\\
    -\infty &\text{otherwise}
  \end{cases}
\end{equation*}
The dual problem involves maximizing $\F_1$ over feasible $\lambda$
and $\mu$, and therefore it will never be advantageous to choose the
dual variables such that $\F_1(\lambda,\mu)=-\infty$ (assuming this
can be avoided).  In fact, provided $\lambda$ satisfies the
inequalities
$\sum_{\gamma\in\Gamma}\N(\gamma,e)\lambda(\gamma)\le\sigma(e)$ (e.g.,
$\lambda\equiv 0$), a suitable nonnegative $\mu$ can always be defined
to make $\F_1(\lambda,\mu)$ finite, whereas, if any of the
aforementioned inequalities is violated in the choice of $\lambda$,
there will be no nonnegative $\mu$ for which $\F_1(\lambda,\mu)$ does not
evaluate to $-\infty$.

Thus, the dual problem can be written as follows.
\begin{equation}\eqlabel{dual-1}
  \begin{split}
    \text{maximize}\quad&
    \sum_{\gamma\in\Gamma}\lambda(\gamma)\\
    \text{subject to}\quad& \sum_{\gamma\in\Gamma}\N(\gamma,e)\lambda(\gamma) \le \sigma(e)\quad \forall e\in E\\
      &\lambda(\gamma)\ge 0\quad\forall\gamma\in\Gamma.
  \end{split}
\end{equation}

\subsection{The case $p=\infty$}

In the $p=\infty$ case, the general Lagrangian is once again not
differentiable.  In this case, an appropriate equivalent problem is
the following.
\begin{equation*}
  \begin{split}
    \text{minimize}\quad& t\\
    \text{subject to}\quad& \sum_{e\in E}\N(\gamma,e)\rho(e)\ge
    1\quad\forall\gamma\in\Gamma\\
    & 0\le \rho(e)\le t\quad\forall e\in E.
  \end{split}
\end{equation*}
The associated Lagrangian takes the form
\begin{equation*}
  \begin{split}
    L_\infty(t,\rho,\lambda,\mu,\eta) &:= t +
    \sum_{\gamma\in\Gamma}\lambda(\gamma)\left(1-\sum_{e\in
        E}\N(\gamma,e)\rho(e)\right) - \sum_{e\in E}\mu(e)\rho(e) +
    \sum_{e\in E}\eta(e)\left(\rho(e)-t\right)\\
    &= t\left(1 - \sum_{e\in E}\eta(e)\right)
    + \sum_{e\in E}\rho(e)\left(
      -\sum_{\gamma\in\Gamma}\N(\gamma,e)\lambda(\gamma) - \mu(e) + \eta(e)
    \right)
    + \sum_{\gamma\in\Gamma}\lambda(\gamma).
  \end{split}
\end{equation*}

As in the $p=1$ case, the Lagrangian is affine in the primal variables
$t$ and $\rho$, so the dual objective function
$\F_\infty(\lambda,\mu,\eta)$ is given by
\begin{equation*}
  \F_\infty(\lambda,\mu,\eta) :=
  \begin{cases}
    \sum\limits_{\gamma\in\Gamma}\lambda(\gamma) & \text{if }
    \sum\limits_{e\in E}\eta(e)=1\quad\text{and}\quad
    \sum\limits_{\gamma\in\Gamma}\N(\gamma,e)\lambda(\gamma) + \mu(e) -
    \eta(e)=0\quad \forall e\in E \\
    -\infty & \text{otherwise}
  \end{cases}
\end{equation*}
Similarly to the case $p=1$, this produces the following form of the
dual problem.
\begin{equation}\eqlabel{dual-inf}
  \begin{split}
    \text{maximize}\quad&
    \sum_{\gamma\in\Gamma}\lambda(\gamma)\\
    \text{subject to}\quad& \sum_{\gamma\in\Gamma}\N(\gamma,e)\lambda(\gamma) \le \eta(e)\quad \forall e\in E\\
    &\sum_{e\in E}\eta(e)= 1 \\
    &\lambda(\gamma)\ge 0\quad\forall\gamma\in\Gamma.
  \end{split}
\end{equation}

\section{Connections to common graph theoretic quantities}
\seclabel{graph-theory}

One of the key contributions of the present work is to make explicit
the connections between the concept of $p$-modulus and the graph
theoretic concepts of shortest path, effective conductance, and min-cut.  The present section is devoted to establishing these
connections.

\subsection{Shortest path}

The connection to shortest path is the broadest and simplest result to
establish.  Recall that $\ell(\Gamma)$, defined in \Eqref{ell-Gamma},
is exactly the length of the shortest walk in $\Gamma$ (in terms of hops).

\begin{theorem}\thmlabel{inf-mod}
  Let $\Gamma$ be any family of walks, then
  \begin{equation*}
    \Mod_\infty(\Gamma) = \frac{1}{\ell(\Gamma)}.
  \end{equation*}
\end{theorem}

\begin{proof}
  Let $\gamma'\in\Gamma$ be a walk such that
  $\ell(\gamma')=\ell(\Gamma)=\ell$.  The density $\rho_0 \equiv 1/\ell$
  is admissible, since for any $\gamma\in\Gamma$
  \begin{equation*}
    \ell_{\rho_0}(\gamma) = \sum_{e\in E}\N(\gamma,e)\rho_0(e) 
    = \frac{\ell(\gamma)}{\ell(\gamma')} \ge 1.
  \end{equation*}
  Thus $\Mod_\infty(\Gamma)\le\ell^{-1}$.  On the other hand, if $\rho\in A(\Gamma)$, then
  \begin{equation}
    \eqlabel{rho-max-bound}
    1\le \ell_\rho(\gamma') = \sum_{e\in E}\N(\gamma',e)\rho(e) \le \ell(\gamma')\max_{e\in E}\rho(e),
  \end{equation}
  so every admissible density must satisfy $\E_\infty(\rho)\ge
  \ell^{-1}$.
\end{proof}

\subsection{Effective conductance}

An undirected graph $G=(V,E,\sigma)$ is a model for a resistor network
with edges representing resistors with conductances $\sigma$,
connected at junctions represented by the vertices.  Let $s$ and $t$
be distinct vertices in $V$.  The effective conductance
$C_{\eff}(s,t)$ (the reciprocal of the effective resistance) between
$s$ and $t$ can be found by minimizing the total power
\begin{equation*}
  \text{Power} = \sum_{(x,y)\in E}\sigma(x,y)\left(\phi(x)-\phi(y)\right)^2
\end{equation*}
over all voltage potentials $\phi:V\to\mathbb{R}$ satisfying
$\phi(s)=0$ and $\phi(t)=1$.  The following theorem shows that, for
$1<p<\infty$, the extremal density $\rho^*$ for the modulus of a
connecting family of walks can be related to a generalized voltage
potential. Such a result, in the language of extremal distance and in
the special case $p=2$, first appeared in the work of Duffin
\cite{Duffin1962} (see also \cite[Proposition 6.2]{epcz}). For a
version in metric spaces, see \cite[Theorem 7.31]{heinonen2001}.

\begin{theorem}\thmlabel{density-potential}
  Let $G=(V,E,\sigma)$ be an undirected graph, let
  $\Gamma=\Gamma(s,t)$ be the connecting family of walks from $s$ to
  $t$, two distinct vertices in $V$, and let $1<p<\infty$.  Let
  $\rho^*$ be the extremal density for \Eqref{cvx-general}.  Then
  there exists a vertex potential $\phi^*:V\to\mathbb{R}$ such that
  $\phi^*(s)=0$, $\phi^*(t)=1$, and
  \begin{equation}\eqlabel{density-potential}
    \rho^*(x,y) = |\phi^*(x)-\phi^*(y)|\qquad\forall(x,y)\in E.
  \end{equation}
  Moreover, this $\phi^*$ solves the optimization problem
  \begin{equation*}
    \begin{split}
      \text{minimize}\quad& \sum_{(x,y)\in E}\sigma(x,y)|\phi(x)-\phi(y)|^p\\
      \text{subject to}\quad& \phi(s)=0,\quad \phi(t)=1.
    \end{split}
  \end{equation*}
  When $p=2$, it follows that $\Mod_2(\Gamma(s,t))=C_{\eff}(s,t)$.
\end{theorem}

\begin{proof}
  Let $\rho^*\in A(\Gamma)$ be the extremal density.  Note that
  $\ell_{\rho^*}(\Gamma)=1$, for, if not, the density
  $\rho' = \rho^*/\ell_{\rho^*}(\Gamma)$ is also admissible and has lower
  $p$-energy than $\rho^*$.  Define $\phi^*(s)=0$ and for
  $x\in V\setminus\{s\}$
  \begin{equation*}
    \phi^*(x) = \min_{\gamma\in\Gamma(s,x)}\ell_{\rho^*}(\gamma).
  \end{equation*}
  Since $\Gamma=\Gamma(s,t)$, $\phi^*(t)=1$.  To see that
  \Eqref{density-potential} holds, define $\rho'(x,y) =
  |\phi^*(x)-\phi^*(y)|$ for $(x,y)\in E$.  For any
  $\gamma=sv_2v_3\ldots v_rt\in\Gamma$, 
  \begin{equation*}
    \ell_{\rho'}(\gamma) = |\phi^*(s)-\phi^*(v_2)| + |\phi^*(v_2)-\phi^*(v_3)| +
    \cdots + |\phi^*(v_r)-\phi^*(t)| \ge |\phi^*(t)-\phi^*(s)|=1,
  \end{equation*}
  so $\rho'\in A(\Gamma)$.  Let $(x,y)\in E$ be an edge.  Without loss
  of generality, $\phi^*(x)\le\phi^*(y)$ and $y\ne s$.  We claim that
  \begin{equation}\eqlabel{rhop-lt-rhos}
    0 \le \rho'(x,y) = \phi^*(y)-\phi^*(x) \le \rho^*(x,y).
  \end{equation}
  The case $x=s$ is trivial.  Suppose that $x\ne s$ and let
  $\gamma\in\Gamma(s,x)$ be a walk such that
  $\ell_{\rho^*}(\gamma)=\phi^*(x)$.  Letting $\gamma'\in\Gamma(s,y)$
  be the walk obtained by first traversing $\gamma$ and then
  traversing the edge $(x,y)$, we have
  \begin{equation*}
    \phi^*(y) \le \ell_{\rho^*}(\gamma') = \ell_{\rho^*}(\gamma) + \rho^*(x,y),
  \end{equation*}
  which implies \Eqref{rhop-lt-rhos}.

  \Eqref{rhop-lt-rhos} implies that $\E_p(\rho')\le\E_p(\rho^*)$ and,
  by uniqueness of the extremal density, $\rho^*=\rho'$.  To see that
  $\phi^*$ solves the optimization problem, let $\phi':V\to\mathbb{R}$
  with $\phi'(s)=0$ and $\phi'(t)=1$ and define
  $\rho'(x,y)=|\phi'(x)-\phi'(y)|$.  As before, $\rho'\in A(\Gamma)$
  and so
  \begin{equation*}
    \sum_{(x,y)\in E}\sigma(x,y)|\phi^*(x)-\phi^*(y)|^p
    = \E_p(\rho^*) \le \E_p(\rho') = 
    \sum_{(x,y)\in E}\sigma(x,y)|\phi'(x)-\phi'(y)|^p.
  \end{equation*}
\end{proof}

\subsection{Max-Flow Min-Cut}

Modulus is also closely related to the Max-Flow Min-Cut
Theorem~\cite{Ford1956}.  There are a number of different ways to see
this; we shall focus on the connection to Ford and Fulkerson's
original work.  Let $G=(V,E,\sigma)$ be a weighted, undirected graph
and let $s,t\in V$ be distinct vertices.  Let $\Gamma$ be the family
of simple paths from $s$ to $t$.  ($\Gamma\subset\Gamma(s,t)$ is an
essential subfamily in the sense of \thmref{Gamma-star}.)  In the
terminology of Ford and Fulkerson, the positive edge function $\sigma$
is called the \emph{capacity} of the edge. A \emph{flow} can be
thought of as a function $\lambda:\Gamma\to[0,\infty)$ with the
property that
\begin{equation*}
  \sum_{\gamma\in\Gamma}\N(\gamma,e)\lambda(\gamma) \le \sigma(e)\qquad\forall e\in E,
\end{equation*}
and the \emph{value} of such a flow the sum of $\lambda(\gamma)$ over
all $\gamma\in\Gamma$.  Thus, Ford and Fulkerson's maximal flow
problem is equivalent to \Eqref{dual-1} and forms a lower bound for
the $1$-modulus of $\Gamma$.

A \emph{disconnecting set} is a subset of edges $D\subseteq E$ with
the property that each $\gamma\in\Gamma$ traverses at least one edge
in $D$, and a \emph{cut} is a disconnecting set that contains no other
disconnecting sets as proper subsets.  The \emph{value} of a
disconnecting set $D$, denoted $v(D)$, is defined as the sum of the
capacities of all edges in $D$.  The Max-Flow Min-Cut Theorem is
stated as follows.
\begin{theorem}[Max-Flow Min-Cut]
  The values of the maximal flow and minimal cut on any graph are
  equal.
\end{theorem}
The connection between this and the $1$-modulus is established in the
following theorem.
\begin{theorem}\thmlabel{1-mod}
  Let $G=(V,E,\sigma)$ be a weighted, undirected graph, let $s,t\in V$
  be two distinct vertices and let $\Gamma(s,t)$ be the connecting
  family of walks between $s$ and $t$.  Then $\Mod_1(\Gamma(s,t))$ is
  equal to the value of maximum flow (and minimum cut) from $s$ to
  $t$.
\end{theorem}

\begin{proof}
  Let $\Gamma\subset\Gamma(s,t)$ be the family of simple paths
  connecting $s$ to $t$.  $\Gamma$ is an essential subfamily and so
  $\Mod_1(\Gamma(s,t))=\Mod_1(\Gamma)$, which is bounded below by the
  value of the maximum flow, as we saw above.

  Let $D\subseteq E$ be a disconnecting set of minimal value $v(D)$
  and define the edge density $\rho$ as
  \begin{equation*}
    \rho(e) :=
    \begin{cases}
      1 & \text{if }e\in D\\
      0 & \text{otherwise}
    \end{cases}
  \end{equation*}
  Since $D$ is a disconnecting set, $\rho\in A(\Gamma(s,t))$.
  Moreover,
  \begin{equation*}
    \Mod_1(\Gamma(s,t)) \le \E_1(\rho) = \sum_{e\in E}\sigma(e)\rho(e) = \sum_{e\in D}\sigma(e) = v(D),
  \end{equation*}
  providing the opposite inequality.
\end{proof}

\section{Dependence on the parameter $p$}\seclabel{p-dependence}

The results of the preceding section are that the $p$-modulus can be
thought of as a generalization of shortest path ($p=\infty$),
effective conductance ($p=2$), and max-flow or min-cut ($p=1$).
Another way to see this is to consider an undirected, unweighted graph
$G$ and let $\Gamma=\Gamma(s,t)$ be the connecting family of walks
between distinct vertices $s$ and $t$.  (Recall the example in
\secref{example}.) In this setting, $\Mod_\infty(\Gamma)$ equals to
the reciprocal of the length of the shortest path in $\Gamma$, while
$\Mod_1(\Gamma)$ equals to the maximum number of edge-disjoint paths
from $s$ to $t$ (a consequence of \thmref{1-mod} and Menger's
theorem).  $\Mod_2(\Gamma)$ strikes a balance between the other two,
in a sense preferring both shortness and abundance in the collection
of paths connecting $s$ to $t$.

In fact, as the present section demonstrates, the parameter $p$ can be
thought of as a continuous ``tuning parameter'', which balances the
sensitivity of the modulus between walk length and walk abundance.
More precisely, given any graph $G=(V,E,\sigma)$ and any nonempty
family of walks $\Gamma$, the function $p\mapsto\Mod_p(\Gamma)$ is
continuous and decreasing.  Moreover, $\Mod_p(\Gamma)^{1/p}$ converges
to $\Mod_\infty(\Gamma)$ as $p\to\infty$.  In fact, with the correct
normalization, the convergence is monotone (see~\thmref{convergence-infinity}).

We begin with bounds on the modulus.

\begin{lemma} \lemlabel{mod-bounds}

  Let $G=(V,E,\sigma)$ be a graph with weights $\sigma>0$.  Let the
  family $\Gamma$ and the parameter $1\le p<\infty$ be given. Then
  \begin{equation*}
    \frac{\sigma_{\min}}{\ell^p} \le \Mod_p(\Gamma) \le \frac{\sigma(E)}{\ell ^p},
  \end{equation*}
  where $\ell:=\ell(\Gamma)$, 
  \begin{equation*}
    \sigma_{\min}:=\min_{e\in E}\sigma(e),\quad\text{and}\quad
    \sigma(E) := \sum_{e\in E}\sigma(e).
  \end{equation*}
\end{lemma}

\begin{proof}
  The upper bound can be obtained by considering the edge density
  $\rho\equiv\frac{1}{\ell}$.  This density is admissible, because for
  all $\gamma\in\Gamma$,
  $\ell_\rho(\gamma) = \frac{1}{\ell}\ell(\gamma) \ge 1$.  Thus
  \begin{equation*}
    \Mod_p(\Gamma) \le \E_p(\rho) = \sum_{e\in E}\sigma(e)\frac{1}{\ell^p}
    = \frac{\sigma(E)}{\ell^p}.
  \end{equation*}
  
  The lower bound can be obtained as follows.  Let $\gamma'\in\Gamma$ such that $\ell(\gamma')=\ell(\Gamma)=\ell$.  By \Eqref{rho-max-bound}, for every $\rho\in A(\Gamma)$ there exists an edge $\tilde{e}\in E$ such that $\rho(\tilde{e})\ge\ell^{-1}$.  Thus,
  \begin{equation*}
    \E_p(\rho) = \sum_{e\in E}\sigma(e)\rho(e)^p \ge \sigma_{\min}\rho(\tilde{e})^p \ge 
    \frac{\sigma_{\min}}{\ell^p}.
  \end{equation*}
\end{proof}

\begin{theorem}\thmlabel{convergence-infinity}
  Let $G=(V,E,\sigma)$ and $\Gamma$ be given.  Then
  \begin{equation}\eqlabel{mod-p-limit}
    \lim_{p\to\infty} \Mod_p(\Gamma)^{\frac{1}{p}} = \frac{1}{\ell(\Gamma)} = \Mod_\infty(\Gamma).
  \end{equation}
  Moreover, the following monotonicity properties holds for $1\le p<q<\infty$:
\begin{gather}
\eqlabel{regular-monotonicity}
\Mod_q(\Gamma)\leq \Mod_p(\Gamma),\\
\eqlabel{holders-mod}
\sigma(E)^{-1/p}\Mod_p(\Gamma) ^{1/p}\leq \sigma(E)^{-1/q}\Mod_q(\Gamma)^{1/q}.
\end{gather}
\end{theorem}
\begin{remark}
Observe that \eqref{eq:mod-p-limit} and \eqref{eq:holders-mod} combine to yield
  \begin{equation}\eqlabel{mod-p-limit-normalized}
    \sigma(E)^{-1/p}\Mod_p(\Gamma)^{1/p}\uparrow\Mod_\infty(\Gamma)\quad\text{as }p\to\infty.
  \end{equation}
\end{remark}

\begin{proof}
  Taking the $p$th root of the inequalities in \lemref{mod-bounds}
  proves \eqref{eq:mod-p-limit}.  \Eqref{regular-monotonicity} follows
  from the fact that for any $\rho\in A^*(\Gamma)$ we have
  $0\le\rho\le 1$, and since $q > p$, this implies that
  \begin{equation}
    \eqlabel{rho-q-lt-rho-p}
    \E_q(\rho) = \sum_{e\in E}\sigma(e)\rho(e)^q \le \sum_{e\in E}\sigma(e)\rho(e)^p = \E_p(\rho).
  \end{equation}
  To prove monotonicity in~\eqref{eq:holders-mod} we use H\"older's
  inequality. Given $\rho\in A^*(\Gamma)$, assume that $1\le
  p<q<\infty$. Then, using the conjugate exponents $q/p$ and $q/(q-p)$,
  \begin{equation}
    \eqlabel{holder-energy-est}
    \begin{split}
      \E_p(\rho) &= \sum_{e\in E}\sigma(e)\rho(e)^p = \sum_{e\in E}
      \sigma(e)^{p/q}\rho(e)^p\cdot\sigma(e)^{1-p/q} \\ & \le
      \left(\sum_{e\in E}\sigma(e)\rho(e)^q\right)^{p/q}
      \left(\sum_{e\in E}\sigma(e)\right)^{1-p/q} =
      \sigma(E)^{1-p/q}\E_q(\rho)^{p/q}.
    \end{split}
  \end{equation}
This can be rewritten as
\begin{equation}\eqlabel{holders-energy}
\sigma(E)^{-1/p}\E_p(\rho) ^{1/p}\leq \sigma(E)^{-1/q}\E_q(\rho) ^{1/q}.
\end{equation}
Picking $\rho$ to be extremal for $\Mod_q(\Gamma)$ we get~\eqref{eq:holders-mod}.
\end{proof}

\begin{theorem}\thmlabel{mod-cont}
  Let $G=(V,E,\sigma)$ and $\Gamma$ be given.  Then the the function
  \begin{equation*}
    p\mapsto \Mod_p(\Gamma)
  \end{equation*}
  is continuous for $1\le p < \infty$.
\end{theorem}

\begin{proof}
If we fix $p$ and let $q\downarrow p$, from \eqref{eq:regular-monotonicity} we get
\[
\limsup_{q\downarrow p}\Mod_q(\Gamma) \leq \Mod_p(\Gamma),
\]
and from \eqref{eq:holders-mod} we get
\[
\liminf_{q\downarrow p}\Mod_q(\Gamma)\geq \liminf_{q\downarrow p}\sigma(E)^{1-q/p}\Mod_p(\Gamma)^{q/p}=\Mod_p(\Gamma).
\]
So $\lim_{q\downarrow p}\Mod_q(\Gamma)=\Mod_p(\Gamma)$.

Likewise, if $q\uparrow p$, then from \eqref{eq:regular-monotonicity} we get
\[
\liminf_{q\uparrow p} \Mod_q(\Gamma) \geq \Mod_p(\Gamma),
\]
and from \eqref{eq:holders-mod} we get
\[
\limsup_{q\uparrow p} \Mod_q(\Gamma) \leq \limsup_{q\uparrow p} \sigma(E)^{1-q/p}\Mod_p(\Gamma)^{q/p}=\Mod_p(\Gamma).
\]
So $\lim_{q\rightarrow p}\Mod_q(\Gamma)=\Mod_p(\Gamma)$.
\end{proof}

\section{Properties of extremal densities}\seclabel{rho-properties}

We end with two properties of the extremal density in modulus
problems: it is continuous as a function of $p$, and it is closely
related to the gradient of the modulus with respect to the weights
$\sigma$.

\subsection{Continuity of the extremal density}
This continuity result arises from Clarkson's inequalities.
(See Theorem~2 of~\cite{Clarkson1936}.)  For a given $p\ge 1$, let
$\|\cdot\|_p$ represent the standard $p$-norm on $\mathbb{R}^m$,
defined on a density $\rho:E\to\mathbb{R}$ as
\begin{equation*}
  \|\rho\|_p := \left(\sum_{e\in E}|\rho(e)|^p\right)^{\frac{1}{p}}.
\end{equation*}
In this context, the relevant Clarkson inequalities take the following
form.
\begin{theorem}\thmlabel{clark}
  Let $p>1$,  $p'=p/(p-1)$,  $x,\ y\in\mathbb{R}^m$, and  $\|\cdot\|_p$ be the standard $p$-norm. 

\noindent If
  $p\ge 2$, then
  \begin{equation*}
    \left\|\frac{x+y}{2}\right\|_p^p + \left\|\frac{x-y}{2}\right\|_p^p \le
    \frac{1}{2}\|x\|_p^p + \frac{1}{2}\|y\|_p^p.
  \end{equation*}
  If $1<p\le 2$, then
  \begin{equation*}
    \left\|\frac{x+y}{2}\right\|_p^{p'} + \left\|\frac{x-y}{2}\right\|_p^{p'} \le
    \left(\frac{1}{2}\|x\|_p^p + \frac{1}{2}\|y\|_p^p\right)^{\frac{p'}{p}}.
  \end{equation*}
\end{theorem}

A first consequence of Clarkson's inequalities is that admissible
densities that are almost-minimizers of the energy are close to the
extremal density in $p$-norm.

\begin{lemma}\lemlabel{clark}
  Let $G=(V,E,\sigma)$ be a graph with $0<\sigma_{\min}\le\sigma$, and
  let $\Gamma$ be a family of walks on $G$. Assume $p>1$,   and let $\rho^*$ be the
  unique extremal density for the $p$-modulus of $\Gamma$. 
  Then, given any admissible density $\rho\in A(\Gamma)$ the following holds.

\noindent If $p\ge 2$,
  then
  \begin{equation}
    \eqlabel{clark-gt2}
    \|\rho-\rho^*\|_p^p\le\frac{2^{p-1}}{\sigma_{\min}}\left(
      \E_p(\rho)-\Mod_p(\Gamma)
    \right).
  \end{equation}
  If $1<p\le 2$, then
  \begin{equation}
    \eqlabel{clark-lt2}
    \|\rho-\rho^*\|_p^{p'} \le \left(\frac{2^p}{\sigma_{\min}}\right)^{\frac{p'}{p}}\left[
      \left(
        \Mod_p(\Gamma) + \frac{1}{2}\big(
          \E_p(\rho) - \Mod_p(\Gamma)
        \big)
        \right)^{\frac{p'}{p}}
        - \Mod_p(\Gamma)^{\frac{p'}{p}}
    \right].
  \end{equation}
\end{lemma}

\begin{proof}
  It is convenient to define two additional edge densities, $f$ and
  $f^*$, as
  \begin{equation*}
    f(e) := \sigma(e)^{\frac{1}{p}}\rho(e)\qquad\text{and}\qquad
    f^*(e) := \sigma(e)^{\frac{1}{p}}\rho^*(e).
  \end{equation*}
  Then
  \begin{equation}
    \eqlabel{norm-fs}
    \|f\|_p^p = \E_p(\rho)\qquad\text{and}\qquad
    \|f^*\|_p^p = \E_p(\rho^*)=\Mod_p(\Gamma),
  \end{equation}
  and
  \begin{equation}
    \eqlabel{diff-f-diff-rho}
    \left\|\frac{f-f^*}{2}\right\|_p^p = \frac{1}{2^p}\sum_{e\in E}\sigma(e)|\rho(e)-\rho^*(e)|^p
    \ge \frac{\sigma_{\min}}{2^p}\|\rho-\rho^*\|_p^p.
  \end{equation}
  Also,
  \begin{equation}
    \eqlabel{sum-f}
    \left\|\frac{f+f^*}{2}\right\|_p^p = \sum_{e\in E}\sigma(e)\left|
      \frac{\rho(e)+\rho^*(e)}{2}
    \right|^p = \E_p\left(\frac{\rho(e)+\rho^*(e)}{2}\right) \ge \Mod_p(\Gamma),
  \end{equation}
  where the final inequality follows from the fact that the
  admissible set $A(\Gamma)$ is convex, so that $(\rho+\rho^*)/2\in
  A(\Gamma)$.

  When $p\ge 2$, Equations~\eqref{eq:norm-fs} through~\eqref{eq:sum-f}
  combined with \thmref{clark} show that
  \begin{equation*}
    \frac{\sigma_{\min}}{2^p}\|\rho-\rho^*\|_p^p \le 
    \frac{1}{2}\E_p(\rho)+\frac{1}{2}\Mod_p(\Gamma)
    - \Mod_p(\Gamma),
  \end{equation*}
  which proves~\eqref{eq:clark-gt2}.  When $1<p\le 2$, a similar
  argument shows that
  \begin{equation*}
    \left(\frac{\sigma_{\min}}{2^p}\right)^{\frac{p'}{p}}\|\rho-\rho^*\|_p^{p'} \le
    \left(\frac{1}{2}\E_p(\rho)+\frac{1}{2}\Mod_p(\Gamma)\right)^{\frac{p'}{p}}
    - \Mod_p(\Gamma)^{\frac{p'}{p}},
  \end{equation*}
  implying~\eqref{eq:clark-lt2}.
\end{proof}

\begin{theorem}
  Let $\Gamma$ be a family of walks on a graph $G$ and let
  $1<p<\infty$.  Then
  \begin{equation*}
    \lim_{q\to p}\|\rho_p-\rho_q\|=0,
  \end{equation*}
  where $\rho_p$ and $\rho_q$ are the unique extremal densities for
  $p$ and $q>1$ respectively.
\end{theorem}

Note that the norm in the theorem is irrelevant, since all norms on
$\mathbb{R}^m$ are equivalent.

\begin{proof}
  By \lemref{clark}, it is enough to show that $\E_p(\rho_q)\to\Mod_p(\Gamma)$ as $q\to p$.  For $q > p$, \Eqref{holder-energy-est} shows that
  \begin{equation*}
    \Mod_p(\Gamma) \le \E_p(\rho_q) \le \sigma(E)^{1-\frac{p}{q}}\E_q(\rho_q)^{\frac{p}{q}}
    = \sigma(E)^{1-\frac{p}{q}}\Mod_q(\Gamma)^{\frac{p}{q}}
  \end{equation*}
  and for $q<p$, \Eqref{rho-q-lt-rho-p} shows that
  \begin{equation*}
    \Mod_p(\Gamma) \le \E_p(\rho_q) \le \E_q(\rho_q) = \Mod_q(\Gamma).
  \end{equation*}
  Then, by continuity (\thmref{mod-cont}), letting $q$ tend to $p$ yields the result.
\end{proof}

\subsection{Extremal density as a gradient}\seclabel{gradient}

We now turn our attention to the behavior of the modulus as a function
of the weights $\sigma:E\to(0,\infty)$.  For this discussion, we fix a
graph $G=(V,E)$ and a non-empty family of walks $\Gamma$, but allow
the edge weights, $\sigma$, to vary.  In order to make explicit the
dependence on the weights $\sigma$, we use the notation
$\E_p(\rho;\sigma)$ and $\Mod_p(\Gamma;\sigma)$ to denote the
$p$-energy of $\rho$ and the $p$-modulus of $\Gamma$ respectively.  We
restrict attention to $p<\infty$, since $\E_\infty$ and $\Mod_\infty$
do not depend on $\sigma$. The next two lemmas show that modulus is
continuous and concave in $\sigma$.

\begin{lemma}\lemlabel{sigma-continuous}
  Let $1\le p<\infty$.  The function
  $\sigma\mapsto\Mod_p(\Gamma;\sigma)$ is Lipschitz continuous.
\end{lemma}

\begin{proof}
  Let $\sigma_1,\sigma_2:E\to(0,\infty)$ and let $\rho_1$ and
  $\rho_2$ be extremal densities for $\Mod_p(\Gamma;\sigma_1)$ and
  $\Mod_p(\Gamma;\sigma_2)$ respectively.  Without loss of generality,
  assume $\Mod_p(\Gamma;\sigma_1)\le \Mod_p(\Gamma;\sigma_2)$.  From
  \lemref{rho-bounds}, it follows that $|\rho_1|\le 1$, so
  \begin{equation*}
    \begin{split}
      \Mod_p(\Gamma;\sigma_2) - \Mod_p(\Gamma;\sigma_1) &\le
      \E_p(\rho_1;\sigma_2) - \E_p(\rho_1;\sigma_1) = \sum_{e\in
        E}\sigma_2(e)\rho_1(e)^p - \sum_{e\in E}\sigma_1(e)\rho_1(e)^p \\
      & = \sum_{e\in E}(\sigma_2(e)-\sigma_1(e))\rho_1(e)^p \le
      \sum_{e\in E}|\sigma_2(e)-\sigma_1(e)| \le
      m\|\sigma_2-\sigma_1\|_{\infty},
    \end{split}
  \end{equation*}
  where $\|\cdot\|_\infty$ is the standard max norm on $\mathbb{R}^m$.
\end{proof}

\begin{lemma}
  Let $1\le p<\infty$.  The function
  $\sigma\mapsto\Mod_p(\Gamma;\sigma)$ is concave.
\end{lemma}

\begin{proof}
  Let $\sigma_1,\sigma_2:E\to(0,\infty)$ and let $\theta\in[0,1]$.
  Let $\rho_1$, $\rho_2$ and $\rho^*$ be extremal densities for
  $\sigma_1$, $\sigma_2$ and $\sigma^*:=\theta\sigma_1 +
  (1-\theta)\sigma_2$, respectively.  Then
  \begin{equation*}
    \begin{split}
      \theta\Mod_p(\Gamma;\sigma_1) +
      (1-\theta)\Mod_p(\Gamma;\sigma_2) &\le \theta\sum_{e\in
        E}\sigma_1(e)\rho^*(e)^p + (1-\theta)\sum_{e\in
        E}\sigma_2(e)\rho^*(e)^p \\
      &= \sum_{e\in E}\sigma^*(e)\rho^*(e)^p = \Mod_p(\Gamma;\theta\sigma_1+(1-\theta)\sigma_2).
    \end{split}
  \end{equation*}
\end{proof}

For $p>1$, the following theorem provides an interpretation of the
unique extremal density in terms of the gradient of
$\Mod_p(\Gamma;\sigma)$ with respect to $\sigma$.

\begin{theorem}
  Let $1<p<\infty$ and define the function
  $F(\sigma):=\Mod_p(\Gamma;\sigma)$ on weights
  $\sigma:E\to(0,\infty)$.  Given a weight $\sigma$, let
  $\rho_\sigma^*$ denote the unique extremal density for
  $\Mod_p(\Gamma;\sigma)$.  Then 
  \begin{equation}\eqlabel{derivative}
    \frac{\partial F(\sigma)}{\partial\sigma(e)} = \rho_\sigma^*(e)^p.
  \end{equation}
\end{theorem}
\begin{proof}
  Let $\sigma:E\to(0,\infty)$ and $\eta:E\to\mathbb{R}$.  For $h>0$
  sufficiently small, $\sigma_h := \sigma+h\eta$ is positive.  Let
  $\rho_h^*$ be the extremal density for $\Mod_p(\Gamma;\sigma_h)$,
  and let  $\rho^*=\rho_0^*$, the extremal density for
  $\Mod_p(\Gamma;\sigma)$.  Note that, for any $\rho:E\to\mathbb{R}$,
  \begin{equation*}
    \E_p(\rho;\sigma_h) = \sum_{e\in E}(\sigma(e)+h\eta(e))|\rho(e)|^p
    = \E_p(\rho;\sigma) + h\sum_{e\in E}\eta(e)|\rho(e)|^p.
  \end{equation*}
  In particular
  \begin{equation*}
    \begin{split}
      \Mod_p(\Gamma;\sigma_h) \le \E_p(\rho^*;\sigma_h) =
      \E_p(\rho^*;\sigma) + h\sum_{e\in E}\eta(e)\rho^*(e)^p =
      \Mod_p(\Gamma;\sigma) + h\sum_{e\in E}\eta(e)\rho^*(e)^p
    \end{split}
  \end{equation*}
  and
  \begin{equation*}
    \Mod_p(\Gamma;\sigma_h) = \E_p(\rho_h^*;\sigma_h)
    = \E_p(\rho_h^*;\sigma) + h\sum_{e\in E}\eta(e)\rho_h^*(e)^p
    \ge \Mod_p(\Gamma;\sigma) + h\sum_{e\in E}\eta(e)\rho_h^*(e)^p.
  \end{equation*}
  Combined, these two inequalities imply that for $h>0$ sufficiently
  small
  \begin{equation}\eqlabel{diff-quotient}
    \sum_{e\in E}\eta(e)\rho_h^*(e)^p \le \frac{F(\sigma+h\eta) - F(\sigma)}{h}
    \le \sum_{e\in E}\eta(e)\rho^*(e)^p.
  \end{equation}
  Now, since each $\rho_h^*\in A^*(\Gamma)$ by \lemref{rho-bounds},
  \begin{equation*}
    \begin{split}
      \E_p(\rho_h^*;\sigma) - \Mod_p(\Gamma;\sigma) &=
      \E_p(\rho_h^*;\sigma_h) - h\sum_{e\in E}\eta(e)\rho_h^*(e)^p -
      \Mod_p(\Gamma;\sigma) \\ &\le
      |\Mod_p(\Gamma;\sigma_h)-\Mod_p(\Gamma;\sigma)| + h\sum_{e\in
        E}|\eta(e)|.
    \end{split}
  \end{equation*}
  \lemref{sigma-continuous} implies that the right-hand side vanishes
  in the limit as $h\to 0$, so
  $\E_p(\rho_h^*;\sigma)\to\Mod_p(\Gamma;\sigma)$.  And \lemref{clark}
  then implies that $\rho_h^*\to\rho^*$.  Taking the limit
  in~\eqref{eq:diff-quotient} yields
  \begin{equation*}
    \lim_{h\to 0^+} \frac{F(\sigma+h\eta) - F(\sigma)}{h} = \sum_{e\in E}\eta(e)\rho^*(e)^p,
  \end{equation*}
  which, since $\eta$ was arbitrary, implies~\eqref{eq:derivative}.
\end{proof}

\section{Related work}\seclabel{related}

The concept of the modulus of a family of walks on a graph is a
natural analog of the continuum definition of modulus of a family of
curves in the plane, introduced by Beurling \cite{beurling1989} and Ahlfors \cite{ahlfors1982}.  The task of
detailing its historical development is complicated by the facts that
modulus seems almost infinitely generalizable and that similar ideas
exist in different fields, with different notation and terminology.
Certainly, the work on maximal flow by Ford and
Fulkerson~\cite{Ford1956} is closely related to modulus (see
\thmref{1-mod}), as is the work on extremal length in networks by
Duffin~\cite{Duffin1962}.  Indeed, Duffin's paper cites Beurling and
Ahlfors as well as Ford and Fulkerson in the introduction, and
Duffin's extremal width is equivalent to a particular modulus problem.
(Compare Theorem~2 of~\cite{Duffin1962} to \thmref{density-potential}
of the present work.)  More recently, one finds the work of
Schramm~\cite{Schramm1993}, which uses a form of modulus to study
square tilings, and the work of Ha{\"i}ssinsky~\cite{Haissinsky2009},
which develops a discrete version of the modulus quite similar to the
definition in the present work, but with a different goal in mind.

\section{Discussion}

This paper presents an exploration, in a fairly general setting, of
the ways in which modulus generalizes and expands upon important graph
theoretical concepts.  Modulus is shown to generalize three
fundamental concepts from graph theory---shortest path, effective
conductance, and min-cut---and to interpolate, in a sense, among these
three concepts.  These results depend upon lower and upper bounds on
the modulus, which are established using tools from convex
optimization.

With these connections firmly in place, a number of questions
naturally arise.  For example, resistance distance has proven quite
valuable in a wide variety of applications.  As we have shown, on a
weighted, undirected graph, this distance is equivalent to the
$2$-modulus, in the sense that the resistance distance between $s$ and
$t$ in $V$ is $\Mod_2(\Gamma(x,y))^{-1}$.  Is $p=2$ the only useful
case?  Or are there applications for which some other value of $p$ is
useful?  Moreover, there is no obvious sense of effective resistance
on directed graphs, but such a concept may well be useful (see,
e.g.,~\cite{Heman}.)  We are aware, for example, of a recent arXiv
paper~\cite{Young} which constructs a notion of resistance distance on
directed graphs that preserves certain important control theoretic
properties.  The average of $\Mod_p(\Gamma(x,y))$ and
$\Mod_p(\Gamma(y,x))$ might also be an interesting quantity to look at
in applications.

There are also several interesting open questions regarding the
dependence of the modulus and extremal density on $p$ beyond the
continuity and monotonicity proved in \secref{p-dependence}.  For
example, it is not yet clear how smooth the function
$p\mapsto\Mod_p(\Gamma)$ is in general.  Moreover, it would be
interesting to better understand the behavior of the unique extremal
density $\rho_p$ either as $p\to 1$ or as $p\to\infty$.  In both
cases, one might wonder if there is a limiting density and, if so, if
the limiting density is extremal for the $1$- or $\infty$-modulus.

Another potentially interesting direction is graph optimization,
similar in spirit to the results of \rref{boyd2008}.  As shown in
\secref{gradient}, the function $\sigma\mapsto\Mod_p(\Gamma;\sigma)$
is concave, so the problem of maximizing the modulus subject to convex
constraints on $\sigma$ (e.g., that the sum of $\sigma$ over all edges
is bounded) is a convex optimization problem whose solution could give
some insight into the structure of a graph relative to a given family
of walks.

As mentioned in the previous section, generalizations of the modulus
presented here seem almost endless.  Some obvious directions of
research include the replacement of edge densities with vertex
densities, or using a mixture of both.  Moreover, the objects in
$\Gamma$ don't actually need to be walks.  One might consider a ``cut
modulus'', for example, where $\Gamma$ is the family of all edge cuts
separating two given sets of vertices.

Another interesting line of inquiry lies in understanding what the
modulus can say about the structure of a graph.  As mentioned in the
introduction, for example, it can be interesting to consider moduli of
families of walks from $s$ to $t$ passing through $u$.  One might also
consider other families, such as the family of simple paths with at
least $L$ hops, or the family of simple cycles with diameter at least
$D$.  What might the moduli of these families say about the structure
of the graph?  What is the expected modulus of these families on a
given random class of graph?

Perhaps even more intriguing is the interpretation offered by the dual
formulation of \secref{lagrange}.  In the dual interpretation, the edge
density $\rho$ is replaced by non-negative function $\lambda$ on
$\Gamma$.  Through duality, $\lambda$ is related to the sensitivity of
the modulus to the admissibility constraints, giving a natural ranking
of the importance of the walks in $\Gamma$.  Those walks with larger
$\lambda$ are more important in the modulus problem than those with
smaller $\lambda$.  This perspective of \emph{most important walks} is
rich for exploration and has been developed in \cite{apc}. For instance, by replacing $\Gamma$ by some properly chosen
subset $\Gamma'$ with large $\lambda$ values, one can
obtain a subfamily with nearly the same modulus as $\Gamma$.  This
viewpoint also lends itself to a concept of sparsification subordinate
to a modulus problem: by choosing a subfamily $\Gamma'$ of most
important walks, and then removing from $E$ all edges that do not
intersect $\Gamma'$, it is possible to produce a subgraph that
approximates the behavior of the original graph from the perspective
of the given modulus problem.

\section*{Acknowledgments}
This material is based upon work supported by the National Science
Foundation under Grant No.~126287 (Albin, Brunner, Perez, Wiens),
through Kansas State University's 2014 Summer Undergraduate
Mathematics Research program, and under Grants No.~1201427
(Poggi-Corradini) and~1515810 (Albin).  The authors are grateful to
Professors Marianne Korten and David Yetter for organizing an
exceptional summer research program.  We also thank the anonymous
referee for a very careful reading and helpful suggestions that
improved the paper.

\section*{References}

\bibliographystyle{acm}
\biboptions{numbers,sort&compress}
\bibliography{2015_ABPPW}

\end{document}